\documentclass
{amsart}
\usepackage{amssymb,amsmath,amsthm,amsfonts,amsopn,url,color,
mathtools,microtype,MnSymbol,
mathrsfs, tikz, tikz-cd}
\usepackage[pagebackref]
{hyperref}
\usepackage{scalerel}
\usepackage{stackengine,wasysym}
\usepackage[shortlabels]{enumitem}
\usepackage{geometry}

\usepackage[normalem]{ulem}
\usepackage[all]{xy}
\input xy
\xyoption{all}
\usepackage{amscd}
\usepackage{soul}

\theoremstyle{plain}
\newtheorem{thm}{Theorem}[subsection]

\newtheorem{prop}[thm]{Proposition}
\newtheorem{lemma}[thm]{Lemma}
\newtheorem{cor}[thm]{Corollary}
\newtheorem{conjecture}{Conjecture}
\newtheorem{defpr}[thm]{Definition/Proposition}

\theoremstyle{definition}
\newtheorem{defn}[thm]{Definition}
\newtheorem*{defn*}{Definition}
\newtheorem*{question*}{Question}

\newtheorem{example}[thm]{Example}
\newtheorem*{example*}{Example}
\newtheorem{rem}[thm]{Remark}
\newtheorem*{rem*}{Remark}

\newtheorem*{nota*}{Notation}

\newcommand{\field}[1]{\mathbb{#1}}
\newcommand{\N}{\field{N}}
\newcommand{\Z}{\field{Z}}

\newcommand{\F}{\field{F}}

\newcommand{\ideal}[1]{\mathfrak{#1}}
\newcommand{\m}{\ideal{m}}

\newcommand{\p}{\ideal{p}}
\newcommand{\q}{\ideal{q}}

\newcommand{\func}[1]{\mathrm{#1} \,}

\newcommand{\Spec}{\func{Spec}}

\newcommand{\hgt}{\func{ht}}

\newcommand{\ra}{\rightarrow}

\newcommand{\be}{\begin{enumerate}}
\newcommand{\ee}{\end{enumerate}}

\newcommand{\li}
 {\leftfootline}

\newcommand{\onto}{\twoheadrightarrow}

\newcommand{\cM}{\mathcal{M}}

\newcommand{\cU}{\mathcal{U}}

\renewcommand{\phi}{\varphi}
\DeclareMathOperator{\Frac}{Frac}

\let\int\relax
\DeclareMathOperator{\int}{i}

\author{Neil Epstein}
\address{Department of Mathematical Sciences \\ George Mason University \\ Fairfax, VA  22030\\
USA}
\email{nepstei2@gmu.edu}

\author{Lorenzo Guerrieri}
\address{Instytut Matematyki \\
Jagiellonian University \\
30-348 Krak\'ow \\ Poland}
\email{lorenzo.guerrieri@uj.edu.pl}
\thanks{Lorenzo Guerrieri is supported by the grant MAESTRO NCN-UMO-2019/34/A/ST1/00263 -- Research in Commutative Algebra and Representation Theory}

\title{An introduction to reciprocal complements of integral domains}
\date{August 5, 2026}
\dedicatory{Dedicated to the memory of Jesse Elliott}

\begin{document}
\begin{abstract}
Given an integral domain $D$ with fraction field $F$, its \emph{reciprocal complement} is the subring of $F$ generated by all $1/d$ for nonzero $d$ in $D$.  This paper serves doubly as a survey of the current state of the field and an update with new results and connections.
\end{abstract}
\maketitle

\section{Introduction}
Let $D$ be an integral domain, and $F$ its fraction field.  The finite sums of the elements $\frac 1d$, $d\in D \setminus \{0\}$, form a subring of $F$, called the \emph{reciprocal complement} $R(D)$ of $D$.

The above is extremely easy to state and prove.  However, this appears to be a new construction, first defined in \cite[Definition 2.1]{nme-Euclidean}.  In this article, we survey what has been determined thus far about this construction, both in general and as applies to polynomial rings, affine coordinate rings of algebraic curves, and other special cases.

The notion first arose as a response to the notion of an \emph{Egyptian} domain, defined and developed in \cite{GLO-Egypt} by the second named author, Loper, and Oman.  If $D$ is a domain, with fraction field $K$, an element $\alpha \in K$ is $D$-\emph{Egyptian} (or just \emph{Egyptian} if the context is clear) if it can be written as a sum of reciprocals of elements of $D$.  Then $D$ is an \emph{Egyptian domain} if every element of $K$ (equivalently, every element of $D$) is Egyptian.  Then in the language of reciprocal complements, $D$ is Egyptian if $R(D)=\Frac D$.  The first named author defined reciprocal complements in \cite{nme-Euclidean} to explore just how far a domain is from being Egyptian.  The subject of that paper was reciprocal complements of Euclidean domains, including $k[x]$ and $\Z$.  The answers were nice and clean -- for such a ring $D$, $R(D)$ is always a field or a DVR.

It was then a logical step for the authors and Alan Loper to investigate $R(D)$ when $D=k[x_1, \ldots, x_n]$, for $n\geq 2$ \cite{nmeGuLo-poly}.  It surprised us to learn that the structure of $R(D)$ in this case is very interesting indeed.  We devote a subsection to this below. In \cite{nmeGuLo-poly} and also in the second named author's article \cite{Gu-recomp}, general properties of $R(D)$ were investigated.  For instance $R(D)$ is always local, and any prime ideal of $R(D)$ is generated by elements of the form $1/f$, $f\in D$ not Egyptian.

Several connections have been made to algebraic geometry, order theory, combinatorics and so forth: Dario Spirito \cite{Spi-recocurve} found interpretations in algebraic geometry for the behavior of $R(D)$ when $D$ is the affine coordinate ring of an algebraic curve.
The paper \cite{Gu-recomp}  made a study of reciprocal complements of semigroup rings for totally ordered pointed semigroups.  In \cite{Elnme-factroids}, Elliott and the first named author found a criterion for elements of $K$ to be in $R(D)$ in terms of a new algebraic structure within $D$ called factroids.

Beyond the introduction you are reading, the rest of this paper is broken up into two sections: the old and the new.  In Section~\ref{sec:prev}, we have attempted to collect most of the salient information on reciprocal complements that one can find in other papers. We found that we can lay it out in a more sensible order, and in some cases prove things in more generality, than we could when we were first learning the subject ourselves.  But this is a growing field, with sprouts in many directions, so we couldn't resist this opportunity to also share some new results, given in Section~\ref{sec:new}.  There, we give connections to the notion of \emph{unit-additivity} (by the first named author and Shapiro \cite{nmeSh-unitadd}), we give a careful accounting of the factor rings of $R(k[x,y])$ by prime ideals, we explore various notions of \emph{irreducibility} of elements of $D$ as reflected in the behavior of $R(D)$, and we provide a one-to-one correspondence between $\Spec R(D)$ and a certain class of subrings of $D$. 

\emph{Notation:} For an integral domain $D$, set $D^\circ := D \setminus \{0\}$, $D^\times = $ the group of units of $D$,  and $\Frac D := $ the fraction field of $D$.

\section{Results from other papers}\label{sec:prev}

The main object of study here is the reciprocal complement of an integral domain: \begin{defn*}
Let $D$ be an integral domain, with fraction field $F$.  The \emph{reciprocal complement} $R(D)$ of $D$ is the subring of $F$ generated by the set $\{\frac 1d\mid d\in D^\circ\}$.

An element $d\in D$ is \emph{Egyptian} if $d\in R(D)$.
\end{defn*}
It is easily seen that $R(D)$ consists merely of all sums of the form $\sum_{i=1}^n \frac 1{d_i}$, $d_i \in D^\circ$, $n \geq 0$.
\subsection{Egyptian and Bonaccian domains}

The first natural question about reciprocal complements is: When is $R(D)$ equal to the entire fraction field of $D$?  Such a domain is called Egyptian.  That is: \begin{defn}[{\cite[Definition 1]{GLO-Egypt}}]
Let $D$ be an integral domain, with fraction field $F$.  We say $D$ is \emph{Egyptian} if $R(D) = F$.
\end{defn}
In fact, Egyptian domains were defined before reciprocal complements were, and with different terminology.  Indeed, the paper \cite{GLO-Egypt} (by the second named author along with Loper and Oman), in preprint form, precedes all of our other work cited here, despite how it may appear from publication years. In the original paper, an Egyptian domain was one where every nonzero $d\in D$ could be written as a sum of reciprocals of \emph{distinct} nonzero elements of $D$.   However, surprisingly, the qualifier `distinct' can be omitted.  That is, we have the following, restated in the language of reciprocal complements: \begin{thm}[{\cite[Theorem 2]{GLO-Egypt}}]
Let $D$ be an integral domain, and let $\alpha \in \Frac D$.  If $\alpha \in R(D)$, then $\alpha$ can be written as a sum of reciprocals of distinct elements of $D$.
\end{thm}
Hence, unlike the study of Egyptian fractions in number theory, the question of distinctness of denominators is a non-issue in the study of reciprocal complements.

It was observed in \cite{GLO-Egypt} that not all Euclidean domains are Egyptian, even though $\Z$ is.  In particular, a polynomial ring $K[x]$ in one variable over a field is not (see \cite[Proposition 1]{GLO-Egypt}).  Accordingly, the first named author made a study of reciprocal complements of Euclidean domains. We need to recall here the definition of unit-additivity.

\begin{defn}[{\cite[Definition 1.1]{nmeSh-unitadd}}]
\label{def:unitadditive}
A commutative ring $R$ is \it unit-additive \rm if the sum of any two units is either a unit or nilpotent.
\end{defn}

The following result holds:
\begin{thm}[{\cite[Theorem 2.10 and Example 2.14]{nme-Euclidean}}]\label{thm:Euclidean}
Let $D$ be a Euclidean domain that is not Egyptian.  Then $D$ is unit-additive, and $R(D)$ is a discrete rank 1 valuation ring.
In the special case $D=K[x]$ for a field $K$, we have $$ R(D) = \left\{ \frac{f}{g} \,\middle\vert\,  f,g \in D, \, g \neq 0, \, \deg(f) \leq \deg(g) \right\} =K[x^{-1}]_{(x^{-1})}. $$ 
For arbitrary non-Egyptian Euclidean domains, one gets an analogous ring, replacing the degree by the appropriate Euclidean function.
\end{thm}

We thus come to the following definition, so named because of the repurposed algorithm of Fibonacci's in the proof of Theorem~\ref{thm:Euclidean}.

\begin{defn}[{\cite[Definition 2.1]{nme-Euclidean}}]
An integral domain $D$ is \emph{Bonaccian} if $R(D)$ is a field or a valuation domain -- that is, if for any $\alpha \in \Frac D$, either $\alpha \in R(D)$ or $\alpha^{-1} \in R(D)$.
\end{defn}

Of course, any Bonaccian domain yields a totally ordered group as the value group of its reciprocal complement.  As seen above, both the trivial group and $\Z$ may be obtained in this way.  In fact, as we record for the first time here, \emph{any} totally ordered Abelian group is the value group of the reciprocal complement of some Bonaccian integral domain.

\begin{prop}
Let $\Gamma$ be a totally ordered Abelian group, written additively.  Let $S = \Gamma_{\geq 0}$ be the submonoid of nonnegative elements.  Let $K$ be any field, and $D = K[x^S]$ the semigroup algebra on $S$.  Then $R(D)$ is a valuation ring with value group $\Gamma$.
\end{prop}
This generalizes \cite[Theorem 4.2]{Gu-recomp}, which established the special case where $\Gamma$ was of finite rank.

\begin{proof}
Any nonzero $f\in D$ is a finite sum of terms of the form $\lambda x^s$, $0\neq \lambda \in K$, $s\in S$.  We define the \emph{degree} $\deg f$ of $f$ to be the largest such $s$. By convention, set $\deg(0) = -\infty$.  Note that for any $f,g \in D$, we have $\deg(fg) = \deg(f)+\deg(g)$, and if $f+g\neq 0$, $\deg(f+g) \leq \max\{\deg(f), \deg(g)\}$.  Set $R := \{\frac fg \mid f,g \in D,\ g \neq 0,\ \deg(f) \leq \deg(g)\}.$ We will show that $R=R(D)$.

First, let us show that $R$ is a ring.  Let $\alpha=\frac fg,\ \beta = \frac{f'}{g'} \in R$, so that $\deg(f) \leq \deg(g)$ and $\deg(f') \leq \deg(g')$.  Then $\deg(ff') = \deg(f) + \deg(f') \leq \deg(g) + \deg(g') = \deg(gg')$, so that $\alpha \beta \in R$.  We have $\alpha + \beta = \frac{fg'+f'g}{gg'}$, and $\deg(fg'+f'g) \leq \max \{ \deg(fg'), \deg(f'g)\} = \max\{\deg(f) + \deg(g'), \deg(f') +\deg(g)\} \leq \deg(g) + \deg(g') = \deg(gg')$, whence $\alpha + \beta \in R$.  Finally, $1=\frac11$, so since $\deg(1) = \deg(1)$, we have $1\in R$.  Thus, $R$ is a ring.

Next, we show that $R(D) \subseteq R$.  Since $R$ is a ring, it is enough to show that any generator $\frac 1d$ of $R(D)$ is in $R$.  But for any $d\in D^\circ$, $\deg(d) \geq 0 = \deg(1)$, so $\frac 1d\in R$.

To see that $R \subseteq R(D)$ (and hence $R=R(D)$), let $\alpha \in R$.  Then $\alpha = \frac fg$, where $s := \deg(f) \leq \deg(g)=:t$.  We have $f= \lambda_s x^s + \sum_{j=1}^n \lambda_j x^{s_j}$, where each $s_j < s$.  But then $f = x^s \left(\lambda_s + \sum_{j=1}^n \frac{\lambda_j}{x^{s-s_j}}\right) = x^s u$, where $u$ is a unit of $R(D)$  (for this fact see Proposition~\ref{pr:gradedarenotegyptian} and  Corollary \ref{cor:units} below). 
Similarly, we can express $g$ as $g=x^t v$ for some unit $v$ of $R(D)$.  Then $\alpha = \frac fg = \frac 1{x^{t-s}} \cdot u\cdot v^{-1}$.  This is then a product of three elements of $R(D)$, whence $\alpha \in R(D)$.

Finally, let us show that $R$ (and hence $R(D)$) is a valuation ring with value group $\Gamma$.  For any nonzero element $\alpha=f/g$ of $R$, set $\nu(\alpha) := \deg g - \deg f$.  Then we always have $\nu(\alpha)\geq 0$, by definition.  Also, for $\alpha,\beta \in R$, one sees easily from the computations that showed $R$ to be a ring that $\nu(\alpha\beta) = \nu(\alpha) + \nu(\beta)$ and $\nu(\alpha+\beta) \geq \min\{\nu(\alpha), \nu(\beta)\}$.  Finally, to see that the value group is $\Gamma$, it suffices to show that $\nu$ is surjective onto $S$.  But for any $s\in S$, we have $\frac 1{x^s} \in R$, and $\nu(1/x^s)=s$.
\end{proof}

Next, consider the following definition: \begin{defn}
An integral domain $D$ is \emph{Egyptian-simple} (or \emph{E-simple}) if any Egyptian element of $D$ is a unit.
\end{defn}

One makes frequent use of the following observations: 

\begin{prop}[{\cite[Proposition 2.2]{nmeGuLo-poly} and \cite[Proposition 8(3)]{GLO-Egypt}}]
Let $D$ be an integral domain, $F$ its fraction field, $E$ the set of its Egyptian elements, and $G=E \cup \{0\}$. Then $G$ is a ring, $R(D)=R(E^{-1}D)$, and the fraction field of $G$ consists of the Egyptian elements of $E^{-1}D$, along with $0$.  In particular, $E^{-1}D$ is E-simple.
\end{prop}

Thus, to study the reciprocal complement of $D$, one may assume $D$ is E-simple by first inverting its Egyptian elements. In particular, this approach led to the following converse to Theorem~\ref{thm:Euclidean}.

\begin{thm}[{\cite[Theorem 6.1]{Gu-recomp}}]\label{thm:Euclconverse}
Let $D$ be an integral domain and $E$ the set of the Egyptian elements of $D$.  If $D$ is not Egyptian, then the following are equivalent: \begin{enumerate}
    \item $R(D)$ is a DVR.
    \item $E^{-1}D$ is a Euclidean domain.
    \item $E^{-1}D \cong k[x]$ for some field $k$.
\end{enumerate}
\end{thm}

Equivalently, if $D$ is an E-simple domain that is not a field, then: \[
R(D) \text{ is a DVR} \iff  D \text{ is Euclidean} \iff D \cong k[x] \text{ for some field } k.
\]

A rich source of E-simple domains can be found in positively graded rings.  In fact, the proof of \cite[Proposition 2.3]{nme-Edom} can be strengthened to the following statement.

\begin{prop}\label{pr:gradedarenotegyptian}
Let $\Gamma$ be a totally ordered abelian group and $\Gamma_{\geq 0}$ the monoid of nonnegative elements.  Let $D = \bigoplus_{\lambda \in \Gamma_{\geq 0}} D_\lambda$ be a $\Gamma_{\geq 0}$-graded domain.  Then for any  element $x \in D$ of positive degree, $x$ is not Egyptian.

In particular, if $D_0$ is a field, then $D$ is E-simple.
\end{prop}

\subsection{Generalities on reciprocal complements}
In this subsection and the next, we will survey what has been written about reciprocal complements in general, before we pass to the special case of polynomial rings.

Let us commence with the following structural information about prime ideals, units, and maximal ideals of $R(D)$.  First we note that any prime ideal is generated by reciprocals:
\begin{prop}[{\cite[Proposition 2.8]{nmeGuLo-poly}}]
Any nonzero  prime ideal of $R(D)$ is generated by elements of the form $\frac 1f$, $f\in D^\circ$.
\end{prop}

More specifically, we have the following: \begin{prop}[{\cite[Proposition 2.7(3)]{Gu-recomp}}]\label{pr:decomposeelement} 
Let $\p\in \Spec R(D)$, and let $\alpha = \sum_{i=1}^n \frac 1{f_i} \in \p$, $0\neq f_i \in D$ such that $\alpha$ cannot be written as a sum of fewer than $n$ elements of the form $\frac 1f$, $f \in D$.  Then $\frac 1{f_i} \in \p$ for all $1\leq i \leq n$.
\end{prop}

\begin{thm}[{\cite[Theorem 2.4]{nmeGuLo-poly}}]\label{thm:islocal}
$R(D)$ is local; its unique maximal ideal is generated by all elements of the form $\frac 1f$ where $f\in D$ is not Egyptian.
\end{thm}

\begin{cor}
\label{cor:units}
    Let $D$ be an E-simple domain.  Let $K$ be the subfield of $D$ consisting of  $0$ and the Egyptian elements of $D$. Then the units of $R(D)$ are precisely those elements of the form $
    u + \sum_{i=1}^n 1/f_i$,
    where $u\in K^\times$ and each $f_i \in D \setminus K$.
\end{cor}

In \cite[Lemma 2.3]{nmeGuLo-poly}, an explicit algorithm  is given for computing the inverse in $R(D)$ of an element of the form $
    u + \sum_{i=1}^n 1/f_i$. The inverse has a similar form: $
    u^{-1} + \sum_{i=1}^m 1/g_i$. The algorithm is recursive and requires computing the inverse of all elements of the form $
    u + \sum_{j=1}^k 1/f_{i_j}$ for every $k< n$, where $1 \leq i_1 < \cdots < i_k \leq n$.
It is easy to show that the inverse of $1 + \frac 1 f$ is $1- \frac{1}{f+1}$. However, the inverse of $1 + \frac 1 x + \frac 1 y$ in $R(K[x,y])$ is a sum of $10$ unit fractions, while the inverse of $1 + \frac 1 x + \frac 1 y + \frac{1}{z}$ in $R(K[x,y,z])$ is a sum of $2110$ unit fractions per our algorithm. Thus, practical use of this algorithm would seem limited.

Next we note that there is a straightforward effect that any localization of $D$ has on $R(D)$: \begin{lemma}[{\cite[Lemma 2.2 and Corollary 2.6]{Gu-recomp}}]
\label{lem:localizations}
Let $D$ be an integral domain and $S$ a multiplicative set. Then $R(S^{-1}D) = R(D)[S]$. In the special case where $D$ is E-simple with maximal subfield $K$ and $S=K[x]^{\circ}$ for some $x \in D$, we have $R(D[x^{-1}])= R(S^{-1}D) = R(D)[x]$.
\end{lemma}

\subsection{Prime ideals and dimension}
Many of our results concern prime ideals and dimension.
For example,
\begin{defpr}[{\cite[Proposition 2.7(1)]{Gu-recomp}}]\label{pr:pxexists}
For any nonzero $x\in D$, there is a unique prime ideal $\p$ of $R(D)$ maximal with respect to avoiding $1/x$. We call this ideal $\p_x$. 
\end{defpr}
It then follows from Lemma \ref{lem:localizations} that $R(D)_{\p_x} = R(D)[x] = R(D[x^{-1}])$.

For a partial converse to  the previous proposition, we have the following, which was originally stated assuming $\dim R(D)<\infty$, but the same proof establishes the following stronger result:
\begin{thm}[{\cite[Theorem 2.12]{Gu-recomp}}]\label{thm:p=px}
Let $\p\in \Spec R(D)$ such that $\dim R(D)/\p <\infty$.  Then there is some $x\in D^\circ$ such that $\p=\p_x$.

If $\dim R(D) <\infty$, it follows that $R(D)$ has nontrivial pseudoradical.\footnote{The \emph{pseudoradical} of a domain is the intersection of all its nonzero prime ideals.}
\end{thm}

Surprisingly, reducing the reciprocal complement of $D$ by a prime ideal corresponds to taking the reciprocal complement of a \emph{subring} of $D$.
\begin{thm}[{\cite[Theorem 2.16]{Gu-recomp}}]\label{thm:modp}
Let $\p$ be a prime ideal of $R(D)$.  Then the set $$L= \left\{ x \in D \setminus \{0\} \,\middle\vert\, \frac{1}{x} \not \in \p \right\} \cup \{0\}$$ is a subring of $D$, and
 the natural composition $R(L) \hookrightarrow R(D) \onto R(D)/\p$ is an isomorphism.
\end{thm}
For an update and improvement on Theorems~\ref{thm:p=px} and \ref{thm:modp}, see subsection~\ref{sub:duality}.

Sometimes, there are strong conditions on dimension.  For example, we have the following, with proof simplified in comparison to the original.
\begin{prop}[{\cite[Corollary 2.9]{Gu-recomp}}]
Suppose $R(D)$ is Noetherian. Then $\dim R(D) \leq 1$.
\end{prop}
\begin{proof}
Since $R(D)$ is Noetherian and local (by Theorem~\ref{thm:islocal}), it has finite Krull dimension. Thus, $R(D)$ has nontrivial pseudoradical by Theorem~\ref{thm:p=px}, and hence is a G-domain by \cite[Theorem 19]{Kap-CR}.  But then by \cite[Theorem 146]{Kap-CR}, $\dim R(D) \leq 1$.
\end{proof}

We next present the following result, which generalizes \cite[Theorem 5.2]{Gu-recomp}.
\begin{thm}
Let $T$ be an integral domain, $I$ an ideal of $T$, $G$ the subring of $T$ consisting of its Egyptian elements and $\{0\}$, and $B$ a subring of $G$.  Set $D=B+I$. Then $R(D) = R(B)+R(I)$,  and if $I \cap G=0$, then  $R(I) \in \Spec R(D)$.
\end{thm}

\begin{proof}
We define $R(I)$ to be all sums of elements of the form $\frac 1i$, $0\neq i\in I$; it is elementary to show that $R(I)$ is an ideal, both of $R(D)$ and $R(T)$.  Since $R(B)$ is a subring of $R(D)$, it follows that $R(B) + R(I) \subseteq R(D)$.

For the opposite inclusion, let $d\in D^\circ$.   Then $d=b+i$ for some $b\in B$ and $i\in I$.  If $b=0$ or $i=0$, then clearly $\frac 1d \in R(B)+R(I)$.  Otherwise, since $b$ is a $T$-Egyptian element, $\frac 1b$ is a unit of $R(T)$.  If $\frac 1i$ is in the maximal ideal of $R(T)$, it follows that $\frac 1b+\frac 1i$ is also a unit of $R(T)$, since $R(T)$ is a local ring. If instead $\frac 1i$ is a unit in $R(T)$, then $i \in G$ and $\frac 1b + \frac 1i \in \Frac G.$ Since $d \neq 0$, we get that $\frac 1b + \frac 1i$ is nonzero and therefore again a unit in $R(T)$.  

Thus, \[
\frac 1d = \frac 1{b+i} = \frac 1i \frac 1b \frac 1{\frac 1b+\frac 1i} \in \frac 1i R(T) \subseteq R(I).
\]
Finally, since any nonzero element of $R(B)$ is a unit of $R(T)$, whereas $R(I)$ is contained in the maximal ideal of $R(T)$ if we assume that $I \cap G =0$ , it follows that $R(B) \cap R(I) = 0$.  Thus, $R(D) / R(I) = R(B) + R(I) / R(I) \cong R(B) / R(B) \cap R(I) \cong R(B)$.  Since the latter is an integral domain, it follows that $R(I)$ is prime in $R(D)$. 
\end{proof}

In general, regarding the Krull dimension of a reciprocal complement, we 
consider the following: \begin{conjecture}\label{con:dim}
For any integral domain $D$, $\dim R(D) \leq \dim D$.
\end{conjecture}
The above conjecture holds whenever $D$ is Egyptian (since then $R(D)$ is a field), or $D$ is a finitely generated $K$-algebra for a field $K$ \cite[Theorem 3.2]{Gu-recomp}, or $D=K[S]$, where $K$ is a field and $S$ is a submonoid of the positive part of a totally ordered abelian group of finite real rank \cite[Theorem 4.10 and Remark 4.11]{Gu-recomp}.

\subsection{Polynomial rings over a field}
We can say quite a bit about $R(D)$, in the special case where $D= D_n^K = K[x_1, \ldots, x_n]$.  First, it has the `expected' dimension.

\begin{thm}[{\cite[Theorem 4.4]{nmeGuLo-poly}}]\label{dimn}
We have $\dim R(D_n^K) = n$.
\end{thm}
Next, it has many prime ideals as long as $n\geq 2$.
\begin{prop}[{\cite[Theorem 6.6]{nmeGuLo-poly}}]\label{pr:inftyprimes} Suppose $n \geq 2$.
Then, for each $1\leq i \leq n-1$, $R(D_n^K)$ has infinitely many primes of height $i$.
\end{prop}

One might wonder about the factorization properties of $R(D_n^K)$.  First of all, factorizations of elements exist. \begin{thm}[{\cite[Theorem 3.11]{nmeGuLo-poly}}]\label{atomic}
Every nonzero nonunit element of $R(D_n^K)$ is a product of irreducible elements. That is, $R(D^K_n)$ is \emph{atomic}.
\end{thm}

However, its factorizations are not well-behaved when $n\geq 2$.
\begin{thm}[{\cite[Theorem 5.8]{nmeGuLo-poly}}]
For any $n\geq 2$, $R(D_n^K)$ is not integrally closed.  Hence, it is not a UFD.
\end{thm}

For more poor behavior, consider the following.

\begin{thm}[{\cite[Corollary 5.7]{nmeGuLo-poly}}]
If $n\geq 2$, $R(D_n^K)$ is not coherent. Hence it is not Noetherian.
\end{thm}
In particular, we show \cite[Theorem 5.6]{nmeGuLo-poly}  that $(1/x_1)R(D_n^K) \cap (1/x_2)R(D_n^K)$ is not a finitely generated ideal.

It is instructive to compare the last two results with Theorem~\ref{thm:Euclidean}, as it highlights a strict dichotomy in behavior between the $n=1$ and $n\geq 2$ cases.  When $n=1$, $R(D_n^K)$ is a DVR, hence a Noetherian, integrally closed, and a UFD.  But when $n\geq 2$, $R(D_n^K)$ satisfies none of these conditions.

An interesting feature of reciprocal complements is that localizing at single elements frequently has the same effect as localizing at prime ideals. In the case of polynomial rings, we have the following:

\begin{thm}[{\cite[Lemma 4.1, Proposition 4.2,  Lemma 4.3]{nmeGuLo-poly}}]
\label{thm:heightofprimes}
Let $[n] := \{1, \ldots, n\}$. For each subset $J$ of $[n]$, set $D_J^K := K[\{x_i \mid i \in J\}]$; thus, $D_n^K = D_{[n]}^K$.  Set $\p_J := \p_{\prod_{i \in J} x_i} = $ the unique prime ideal of $R(D_{n}^K)$ maximal with respect to avoiding the element $\prod_{i\in J} x_i$, as in Proposition~\ref{pr:pxexists}.  Then \[
R(D_{n}^K)\left[\prod_{i\in J} x_i\right] = R(D_{n}^K)_{\p_J} = R\left(D_{[n]\setminus J}^{K(\{x_i \mid i\in J\})}\right).
\]
In particular, \begin{enumerate} \item $R(D_{n}^K)[\prod_{i=1}^n x_i] = K(x_1, \ldots, x_n) = \Frac D_{n}^K$,
\item For any subset $J \subset [n]$, $\hgt \p_J = n-\#(J)$, and
\item For any two subsets $J, J' \subseteq [n]$, we have $J \subseteq J'$ if and only if $\p_{J'} \subseteq \p_J$.  Hence, whenever $J \neq J'$, we have $\p_J \neq \p_{J'}$.
\end{enumerate}
\end{thm}
In the above, it is important to remember that $1/x_i \in R=R(D_n^K)$, so adjoining $x_i$ is inverting an element in $R$.

Here is an exotic property of $R(D_2^K)$.
\begin{thm}[{\cite[Theorem 7.5]{nmeGuLo-poly}}]
Any finitely generated ideal in $R(D_2^K)$ is in all but finitely many prime ideals.
\end{thm}
This is in some sense opposite to how primes behave in a Noetherian ring.  There, any ideal has only finitely many primes minimal over it.  Here, any finitely generated nonzero ideal has \emph{almost all} primes minimal over it.

\subsection{Reciprocal complements of curve singularities}

After understanding the main properties of the reciprocal complements of polynomial rings over a field, a natural next question is to ask how this construction behaves on finitely generated algebras over fields. An interesting answer to this question has been given by Dario Spirito in \cite{Spi-recocurve} in the case of integral domains of Krull dimension one that define curve singularities.

The methods in his work explore a surprising relation between valuation overrings of the reciprocal complement and points in the projective closure of a curve. Let $D$ be a one-dimensional domain finitely generated as a $k$-algebra, where $k$ is an algebraically closed field. A \emph{realization} $X$ of $D$ is an irreducible curve in an affine space $\mathbb{A}_k^n$ such that $D$ is isomorphic to the coordinate ring of $X$.

The projective closure of $X$ in $\mathbb{P}_k^n$ will be denoted by $\overline{X}$. The variety $X$ is \emph{regular at infinity} if all the points of $\overline{X} \setminus X$ are regular. Then:

\begin{thm}[{\cite[Theorem 2.1]{Spi-recocurve}}]
\label{thm:dario}
The following conditions are equivalent:
\begin{enumerate}
\item[(1)] $R(D)$ is not a field ($D$ is not Egyptian). 
\item[(2)] If $X$ is a realization of $D$ that is  regular at infinity, then $|\overline{X} \setminus X| = 1$.
\item[(3)] If $X$ is any realization of $D$ and $\nu : Y \to X$ is a normalization
of $X$, then $|\nu^{-1}(\overline{X} \setminus X)| = 1$.
\end{enumerate}
If these equivalent conditions are satisfied, $R(D)$ is a domain of Krull dimension one whose  integral closure is a DVR (corresponding to the localization of the coordinate ring of $\overline{X}$ at the unique point of $\overline{X} \setminus X$). 
\end{thm}

In the case where $R(D)$ is not a field in the above context, Spirito describes also the value semigroup of $R(D)$  (with respect to the valuation induced by the integral closure) in connection with the Weierstrass semigroups of the curve $X$.

\subsection{Connections with factroids}
The following is a greatly simplified version of some of the material in \cite{Elnme-factroids}.
\begin{defn}
Let $D$ be a domain.  A \emph{factroid} of $D$ is an additive subgroup $H$ of $D$ that is closed under factors -- i.e., such that whenever $f,g \in D^\circ$ with $fg \in H$, we have $f,g \in H$.
\end{defn}
It is elementary that any intersection of factroids of $D$ is a factroid of $D$. This includes the empty intersection, as $D$ is clearly a factroid of itself.  Hence, we may define the following: \begin{defn}
\label{def:principalfactroid}
For any subset $S$ of an integral domain $D$, the \emph{factroid generated by $S$}, denoted $[S]_D$, is the smallest factroid of $D$ containing $S$ -- i.e., the intersection of all factroids of $D$ that contain $S$.
\end{defn}

We may generate $[S]_D$ iteratively as follows. Let $F_1(S)$ be the set of all sums of factors of elements of $S$.  Then for each $n\geq 1$, define $F_{n+1}(S) := F_1(F_n(S))$.  Then $[S]_D = \bigcup_n F_n(S)$.

\begin{defn}
\label{def:regularfactroid}
A factroid $F$ of $D$ is \emph{regular} if for any nonzero $g\in D$, we have $[gF]_D :_D g = F$.  As the regular factroids of $D$ are closed under intersection, we may define for any subset $S$ of $D$ the \emph{regular factroid generated by $S$}, denoted $G(S)$, as the intersection of all the regular factroids of $D$ that contain $S$.
\end{defn}

Another way to characterize $G(S)$ is as follows: $G(S) = \bigcup _{d \in D^\circ} ([dS]_D : d)$.  When $S= \{x\}$ is a singleton, we denote $[S]_D$ by $[x]_D$, $F_n(S)$ by $F_n(x)$, etc.

Our main interest in factroids is the following:
\begin{thm}[{\cite[from Corollary 7.6]{Elnme-factroids}}]
\label{thm:factroids}
Let $D$ be a domain, $K$ its fraction field, and $\alpha \in K^\times$.  The following are equivalent: \begin{itemize}
    \item $\alpha \in R(D)$.
    \item There exist $a,b\in D^\circ$ with $\alpha=a/b$ and $a\in [b]_D$.
    \item There exist $a,b \in D^\circ$ with $\alpha=a/b$ and $a\in G(b)$.
    \item For any $a,b \in D^\circ$ with $\alpha=a/b$, we have $a\in G(b)$.
    \item For any $a,b\in D^\circ$ with $\alpha=a/b$, there is some $c\in D^\circ$ with $ca \in F_1(cb)$.
\end{itemize}
\end{thm}

\section{Some new results}\label{sec:new}

\subsection{Connections with unit-additivity}

In this subsection we describe a connection between $E$-simple domains and unit-additive domains (see Definition \ref{def:unitadditive}). 

\begin{prop}
An integral domain $D$ is unit-additive if and only if for any $d, e, f\in D^\circ$ with $d=\frac 1e + \frac 1f$, $d$ is a unit.  In particular, any E-simple domain is unit-additive.
\end{prop}

\begin{proof}
Suppose $D$ is unit-additive.  Let $d,e,f$ be as in the statement.  Then $(de-1)(df-1)= \frac ef \cdot \frac fe=1$, so $de-1$ is a unit.  Since $D$ is unit-additive and $de\neq 0$, we have that $de=(de-1)+1$ is a unit, whence $d$ is a unit.

Conversely, suppose the implication above holds.  Let $u,v$ be units of $D$ such that $u+v\neq 0$.      Then $u+v = \frac 1{u^{-1}}+ \frac 1{v^{-1}}$, whence $u+v$ is a unit by assumption.  Thus, $D$ is unit-additive.

Finally, any $d$ as in the statement of the theorem is clearly an Egyptian element of $D$. Hence if $D$ is E-simple, any such $d$ must be a unit.
\end{proof}

However, E-simplicity is not \emph{equivalent} to unit-additivity, as the following example shows.  In the following, Justin Chen showed that $S^\times=D^\times$, and Will Sawin proved that $D^\times = k^\times$.

\begin{example}[{\cite{Ch-pers25, Saw-MOthisring}}]\label{ex:ChenSawin}
Let $D:=k[x,y,z]/(xyz-xy-xz-yz)$, where $k$ is any field. This is an integral domain because $xyz-xy-xz-yz$ is an irreducible polynomial.  Then $D$ is a subring of  $A=k[x, \frac 1x, y, \frac 1y, z, \frac 1z]/(\frac 1x + \frac 1y + \frac 1z -1)$.  Applying the change of variables $a=1/x$, $b=1/y$, $c=1/z$, we have $A = k[a, \frac 1a, b, \frac 1b, c, \frac 1c] / (a+b+c-1) \cong k[a\frac 1a,b, \frac 1b,\frac 1{1-a-b}]$, where $c$ maps to $1-a-b$.  Let $B$ be the subring $k[a, \frac 1a, b, \frac 1b]$ of this representation of $A$.  The group $B^\times$ of units of $B$ is $\bigcup_{m,n \in \Z} k^\times a^m b^n$. Since $c=1-a-b$, we have $A=B[1/c]$. 

We claim that $A^\times = \bigcup_{m,n,p \in \Z} k^\times a^m b^n c^p$.  Clearly $c=1-a-b$ is an irreducible element of $B$. Let $\gamma \in A^\times$.  Then there is some $\delta \in A$ such that $\gamma \delta =1$.  Write $\gamma=\frac g{c^s}$, $\delta = \frac h{c^t}$, with $g,h \in B$ and $s,t \in \N_0$.  Then $gh=d^{s+t}$ in $B$, so by irreducibility, there exist $G,H \in B$ and nonnegative integers $i,j$ with $i+j=s+t$ and $g=d^iG$, $h=d^jH$.  Thus, $GH=1$, so that since $G,H \in B$, we have $G = \lambda a^m b^n$ for some $\lambda \in k^\times$ and $m,n \in \Z$.  Thus, $\gamma = \lambda a^m b^n c^{i-s}$, finishing the proof of the claim.

Tracing back through the isomorphisms, it then follows that the units of $A$ are all of the form $\lambda x^m y^n z^p$. So it remains to show that the only such elements in $D$ have $m=n=p=0$.  The inclusion $D \hookrightarrow A$ induces a map $\phi: D/(z-1)D \ra A/(z-1)A$.  But after the identification $z \mapsto 1$ and hence $y \mapsto -x$ in both rings, $\phi$ reduces to the inclusion of $k[x]$ into $k[x^{\pm 1}]$; hence it is injective.  Now, let $\lambda x^m y^n z^p$ be a unit in $D$.  Then its image in $D/(z-1)D \cong k[x]$ is $(-1)^n\lambda x^{m+n}$ but must be a unit in $k[x]$; thus $m=-n$.  By symmetry, we must also have $n=-p$ and $m=-p$.  Hence $m=-m$, whence $m=0$, and by symmetry also $n=p=0$.  Hence, $D^\times = k^\times$.

Now, let $S = k[r,s,t,\frac 1r + \frac 1s + \frac 1t] \cong k[r,s,t,u] / (rstu-rs-rt-st)$.  This is a graded ring, generated over $k$ by homogeneous elements $r,s,t$ of degree 1 and $u$ of degree $-1$.  Since $S$ is a subring of $k[r, \frac 1r, s, \frac 1s, t, \frac 1t]$, whose units are $\lambda r^m s^n t^p$ for $\lambda \in k^\times$ and $m,n,p \in \Z$, $u$ is not a unit.  But by \cite[Proposition 4.10(b)]{BrGu-polybook}, any unit of $S$ must be homogeneous.  The homogeneous elements of negative degree all look like $\lambda u^s$ for $s\in \N_0$, so they are multiples of $u$, and hence cannot be units.  If there were a homogeneous unit of positive degree, its inverse would have negative degree, hence there are none.  Thus, $S^\times = (S_0)^\times$, where $S_0$ is the degree $0$ subring of $S$. $S_0$ is generated as a $k$-algebra by $ru$, $su$, and $tu$. Thus, the $k$-algebra homomorphism $\phi: k[x,y,z] \rightarrow S$ given by $x \mapsto ru$, $y \mapsto su$, $z\mapsto tu$ has image $S_0$.  One sees easily that the kernel of $\phi$ contains the element $xyz-xy-xz-yz$, so $\phi$ induces a surjective map $\psi: D \onto S_0$.  Now, $D$ is an integral domain of dimension 2, and $S_0$ is a domain of dimension at least two, since $(0) \subset (r,s)S \cap S_0 = (ru, su) \subset (r,s,t)S \cap S_0 = (ru, su, tu)$ is a chain of primes of length 2. Since any nontrivial quotient ring of $D$ must have dimension 1 or less, it follows that $\psi$ is an isomorphism.  Thus, $S^\times = S_0^\times \cong D^\times = k^\times$, whence $S$ is unit-additive.

However, $S$ is not E-simple, since $u=\frac 1r + \frac 1s +\frac 1t$ is clearly an Egyptian element of $S$, but as we showed above, it is not a unit.
\end{example}

\subsection{The quotients of $R(K[x,y])$}

The goal of this section is to describe the quotients at prime ideals of the ring $R=R(K[x,y])$. 

From the results in the previous sections we know that $\dim R =2$ and every prime of $R$ is of the form $\p_f$ for some $f \in D = K[x,y]$. The ring $K[x,y]$ is E-simple, hence using Theorem \ref{thm:islocal}, it is straightforward to see that $\p_f = \m$ is the maximal ideal of $R$ if and only $f \in K$. 

In the following we find sufficient conditions on a polynomial $f$ to have $\p_f=(0)$ (i.e. $\frac 1f$ belongs to each nonzero prime of $R$). These conditions allow us to show that all the quotients of $R$ by height one primes are DVRs.

We recall that for linear forms $f$ and for polynomials of the form $f = x^p +y^q$ with $p,q$ coprime, the prime $\p_f $ is nonzero and non-maximal \cite[Lemma 3.6, Lemma 5.5]{nmeGuLo-poly}, and therefore it has height one.

We start with some simple linear algebra remarks about polynomials.

  \begin{rem}
\label{linalgrem}
Let $f_1, \ldots, f_n \in K[t]$ be polynomials in one variable over a field $K$. A standard linear algebra argument shows that if $\deg f_i \leq n-2$ for every $i$, then the set $\{ f_1, \ldots, f_n \}$ is not $K$-linearly independent. Equivalently, also if $f_1, \ldots, f_n \in t K[t]$ and $\deg f_i \leq n-1$ for every $i$, then the set $\{ f_1, \ldots, f_n \}$ is not $K$-linearly independent.
\end{rem}

  \begin{lemma}
\label{linalg1}
Let $f,g \in tK[t]$ be two polynomials in one variable with zero constant term. For every $N \gg 0$ large enough, the set $\{ f^ig^j \}_{0 \leq i,j \leq N}$ is not $K$-linearly independent.
\end{lemma}
 
 \begin{proof}
The set $\{ f^ig^j \}_{0 \leq i,j \leq N}$ has cardinality $(N+1)^2$. The degree of each element of this set is bounded above by $N(\deg f + \deg g)$. Hence, it is enough to choose $N > \deg f + \deg g$ and apply Remark \ref{linalgrem}.
 \end{proof}

Recall that from Definition \ref{def:principalfactroid}, for $g \in D$, the principal factroid $[g]_D$ is the smallest additive subgroup of $D$ containing $g$ and closed under factors. By Theorem \ref{thm:factroids}, we have that $\frac{f}{g} \in R(D)$ if and only if there exists $h \in D$ such that $hf \in [hg]_D$.

  \begin{lemma}
\label{linalg2}
Let $f,g \in D=K[x,y]$ be polynomials with zero constant term. Suppose that $f,g$ are coprime \rm(\it equivalently, they form a regular sequence in $D_{(x,y)}$\rm)\it. Then, for every $N \gg 0$ large enough, $ \frac{x}{(fg)^N}, \frac{y}{(fg)^N} \in R(D).$
\end{lemma}

 \begin{proof}
It is enough to prove the result for $x$. If $x$ divides $f$ or $g$, the result follows immediately. Otherwise, write $f = f_1(y) + xf_2$, $g = g_1(y) + xg_2 $ with $f_1, g_1 \in yK[y] \setminus \{ 0 \} $. By Lemma \ref{linalg1}, there exist $N$ such that the set $\{ f_1^ig_1^j \}_{0 \leq i,j \leq N}$ is not $K$-linearly independent. Hence, there exist elements $u_{ij} \in K$, not all zero, such that $ \sum_{i,j} u_{ij} f_1^i g_1^j =0 $. Write $h = \sum_{i,j} u_{ij} f^i g^j $. Observe that $h \neq 0$ since $K[f,g] \cong K[x,y]$ by \cite[Section 5]{Kap-Rseq}.  Moreover, $h \in [f^Ng^N]_D$ and $h= \sum_{i,j} u_{ij} f_1^i g_1^j + x h' = x h'$ is divisible by $x$. It follows that $x \in [f^Ng^N]_D$, and this implies that $ \frac{x}{(fg)^N} \in R(D)$ by Theorem~\ref{thm:factroids}.
 \end{proof}

  To deal with polynomials with nonzero constant term, we can use the following result:

   \begin{lemma}
\label{lemma:associated}
Let $D$ be an E-simple integral domain with maximal subfield $K$. Let $f \in D \setminus K$ and $u \in K^\times$. Then 
$\frac{1}{f+u}$ and $\frac{1}{f}$ are associates in $R(D)$.
\end{lemma}

\begin{proof}
  Observe that: 
 $$ \frac{1}{f+u} = \frac{1}{f} \left( \frac{1}{1+\frac{1}{u^{-1}f}} \right)  $$ and the element $ 1+\frac{1}{u^{-1}f} $ is a unit in $R(D)$ by Corollary \ref{cor:units}. 
\end{proof}
 
   \begin{cor}
\label{pseudorad}
Let $f \in D=K[x,y]$ be a polynomial having at least two distinct non-associated irreducible factors. Then $\frac{1}{f}  $ is in the pseudoradical of $R(D)$. Equivalently, $\p_f = (0)$.
\end{cor}
 
 \begin{proof}
 By Lemma \ref{lemma:associated}, we can reduce to the case where each factor of $f$ has zero constant term.
It is enough then to show that the reciprocal of the product of two non-associate irreducible polynomials $f,g \in (x,y)D$ is in the pseudoradical of $R(D)$. By Lemma \ref{linalg2}, there exists $N $ such that 
\[
\frac{1}{f^N g^N} \in \left( \frac{1}{x} \right)\cap \left( \frac{1}{y} \right).
\]
Recall now that by Theorem \ref{thm:heightofprimes}, $ \p_x$ and $\p_y$ are distinct height one primes. 
 Since $1/x$ is in every height one prime except $\p_x$, and since $1/y$ is in every height one prime except $\p_y$,  it follows that $ \left( \frac{1}{x} \right)\cap \left( \frac{1}{y} \right) $ is contained in the pseudoradical of $R(D)$. We conclude using the fact that the pseudoradical is a radical ideal.
 \end{proof}

 Before proving the theorem describing the quotients at height one primes, we provide a general sufficient condition that guarantees equality $\p_f=\p_g$ for elements $f,g$.

 \begin{lemma}
\label{lemalgdip}
Let $D$ be an E-simple integral domain with maximal subfield $K$. Suppose that $f,g \in D \setminus K$ are not algebraically independent over $K$. Then, $\p_f = \p_g$. In particular, if $g \in K[f] \setminus K$, then $\p_f = \p_g$.  
\end{lemma}

\begin{proof}
By definition of prime ideal, it is clear that $\p_f = \p_{f^e}$ for every $e \geq 1$. By Lemma \ref{lemma:associated}, $\frac{1}{f} $ and $\frac{1}{f+u}$ are associates in $R=R(D)$ for every $u \in K$.  It follows easily that the reciprocal of any element $g \in K[f] \setminus K$ is associated in $R$ to $\frac{1}{f^e}$, for some $e$. Therefore $\p_g = \p_f$.
 
  Now let $g$  be arbitrary.
By assumption, there exists an equation $ \phi= \sum_{i,j} u_{ij} f^ig^j=0 $ such that the coefficients $u_{ij} \in K$ are not all zero. Since $K[f,g]$ is a domain, we can assume without loss of generality that $\phi$ is a multiple in $D$ of neither $f$ nor $g$. We can then write $ 0= \phi = \phi_1 + g \psi $ with $\phi_1 \in K[f] \setminus K$ (since $g$ is not invertible) and $\psi \in K[f,g] \setminus \{0\}$ (since $f$ is not invertible, hence not algebraic over $K$). By the first part of the proof we have $ \p_{\phi_1} = \p_f $. Thus $\p_{g \psi} = \p_f$, which implies $\p_f \subseteq \p_g$. Switching the roles of $f$ and $g$ in the above argument, we obtain the other inclusion $\p_f \supseteq \p_g$.
\end{proof}

 \begin{thm}
\label{quotientdim2}
Let $\p$ be a height one prime of $R=R(K[x,y])$. Then $ R/\p $ is a DVR.
\end{thm}

\begin{proof}
By Theorem \ref{thm:modp}, $ R/\p  \cong R(L)$ where $L=\{ f \in K[x,y] \mid \frac{1}{f} \not \in \p \}$ is clearly generated as a $K$-algebra by irreducible polynomials.
Since the Krull dimension of $R$ is finite, by Theorem \ref{thm:p=px}, there exists $f \in K[x,y]$ such that $\p=\p_f$. By Lemma \ref{lemalgdip}, we have $K[f] \subseteq L$ and we can assume without loss of generality $f \in (x,y)$. We can also assume that $f$ is irreducible by Corollary \ref{pseudorad} (using also the fact that $\p_f = \p_{f^e}$ for every $e \geq 1$). 

Again by Lemma \ref{lemalgdip}, if we assume $K[f] \subsetneq L$, then there exists some element of $ L \setminus K[f] $ with zero constant term and relatively prime with $f$.
Call such an element $g$. It follows that the product $fg$ satisfies the assumption of Corollary \ref{pseudorad}. Hence, $\frac{1}{fg}$ is in the pseudoradical of $R(D)$ and therefore $\frac{1}{fg} \in \p $, showing $\frac{1}{g} \in \p$  so that $g \not \in L$. This is a contradiction. Thus, $L=K[f]$ and $ R/\p  \cong R(L) = R(K[f])$ is a DVR by Theorem \ref{thm:Euclconverse}
\end{proof}

\begin{rem}
    \label{remark:mirror}
    We point out an interesting comparison between the properties of the regular local ring $K[x,y]_{(x,y)}$ and of the ring $R(K[x,y])$. Both rings are local and two-dimensional, and with respect to some properties they act one as a mirror of the other:
    \begin{itemize}
        \item In $K[x,y]_{(x,y)}$, every nonzero (finitely generated) ideal is contained in only finitely many prime ideals. In $R(K[x,y])$, every finitely generated ideal is contained in all but finitely many prime ideals \cite[Theorem 7.5]{nmeGuLo-poly}. 
        \item In $K[x,y]_{(x,y)}$, the localizations at height one primes are all regular local rings. In $R(K[x,y])$, the localizations at height one primes are one-dimensional and Noetherian but sometimes not regular \cite[Theorem 7.3]{nmeGuLo-poly}, \cite[Example 3.3]{Gu-recomp}. 
        \item In $K[x,y]_{(x,y)}$, the quotients at height one primes are one-dimensional and Noetherian but sometimes not regular. In $R(K[x,y])$, the quotients at height one primes are all regular local rings (by Theorem \ref{quotientdim2}).
    \end{itemize}
\end{rem}

 \subsection{Relations between irreducibility in $D$ and properties of $R(D)$}

In this section we show how conditions on elements of a domain $D$, related with irreducibility, imply conditions on the corresponding prime ideals in $R(D)$. 
 
 In the following $D$ will denote an E-simple domain with maximal subfield $K$. The reciprocal complement $R(D)$ of $D$ will be denoted by $R$ and its maximal ideal by $\m$. Denote by $\overline{R}$ the integral closure of $R$ in its quotient field.

We define now examples of regular factroids associated to overrings of $R(D).$

\begin{defn}\label{def:GAg}
   For $g \in D$ and $A$ an overring of $R$, we define the factroid
 $$ G_{A}(g):= \left\lbrace f \in D \,\middle\vert\, \frac{f}{g} \in A \right\rbrace. $$ 
 Notice that if $A=R(D)$, then $G_A(g)= G(g)$ is the regular factroid generated by $\{g\}$, as in Definition~\ref{def:regularfactroid}.  
\end{defn}

 We have a chain of inclusions
 $$  \langle 1,g \rangle_K \subseteq [g]_D \subseteq G(g) \subseteq  G_{\overline{R}}(g). $$
 
 If $D = \bigoplus_{s \in \Gamma_{\geq 0}} D_s$ is graded with respect to the positive part of a totally ordered abelian group and $D_0=K$, 
 then we also have $$ G_{\overline{R}}(g) \subseteq \lbrace f \in D \mid \deg(f) \leq \deg(g) \rbrace. $$ 
 Indeed, observe that the ring
 $$ W= \left\lbrace \frac{f}{g} \,\middle\vert\, f,g \in D, \ \deg(f) \leq \deg(g) \right\rbrace  $$ is a valuation overring of $R$.   Since $W$ is integrally closed, it also contains $\overline{R}$. Also note that $\m \subseteq \m_W$, as whenever $f \in D \setminus K$ we have $\deg(f) >0$, so that $\m \overline{R} \subseteq \m W \subseteq \m_W$.

In the case $D=K[x_1, \ldots, x_n]$, we have that $\lbrace f \in D \mid  \deg(f) \leq \deg(g) \rbrace$ is a finite dimensional $K$-vector space, and so $G(g) $ and $G_{\overline{R}}(g)$ are finite dimensional $K$-vector spaces as well.
 \medskip
 
 In the graded case, for an element $f \in D$, denote by $f_h$ its homogeneous component of largest degree. We show now that the elements of $\overline{R}$ of the form $\frac fg$ with $\deg(f)= \deg(g)$ are very special.

  \begin{lemma}
\label{hompart} Let $\Gamma$ be a totally ordered abelian group.
Let $D$ be positively $\Gamma$-graded over a field $K$.  Assume $K$ is algebraically closed in $\Frac D$ \rm(\it e.g., if $D=K[x_1, \ldots, x_n]$\rm)\it.
Let $f,g \in D$ be such that $\deg(f)=\deg(g)=\delta \in \Gamma_{\geq 0}$. 
Suppose that $ \frac{f}{g} \in \overline{R}$. Then $f_h = u g_h$ for some $u \in K$.
\end{lemma}
 
 \begin{proof}
Let 
 $$ \left( \frac{f}{g} \right)^e + a_{e-1} \left( \frac{f}{g} \right)^{e-1} + \ldots + a_{1}  \frac{f}{g} + a_0=0$$ be an equation of integral dependence with $a_0, \ldots, a_{e-1} \in R$. For every $i=0, \ldots, e-1$, we can write $a_i=u_i+b_i$ with $u_i \in K$ and $b_i \in \m$. We have
 $$ \left( \frac{f}{g} \right)^e + u_{e-1} \left( \frac{f}{g} \right)^{e-1} + \ldots + u_{1}  \frac{f}{g} + u_0  \in \m \overline{R}  \subseteq \m_W, $$
 since $f/g \in \overline{R}$ and each $b_i \in \m$. 
 
 In particular, $h:=f^e + u_{e-1} f^{e-1}g + \ldots + u_1 fg^{e-1} + u_0 g^e$ is in $D$ and $\frac{h}{g^e} \in \m_W$. Hence, $\deg(h) < e\delta$ and, passing to the homogeneous components of degree $e\delta$, we obtain 
  $f_h^e + u_{e-1} f_h^{e-1}g_h + \ldots + u_1 f_hg_h^{e-1} + u_0 g_h^e=0$. Dividing this last equation by $g_h^e$, we find that $\frac{f_h}{g_h}$ is an element of $\Frac D$ algebraic over $K$. Hence, it must be an element of $K$.
 \end{proof}
 
 \begin{rem}
 \label{remarkmember}
 If $D$ is graded, $K$ is algebraically closed in $\Frac D$, and we are interested in establishing  when a fraction $\frac{f}{g} $ with $f,g \in D$ is in $R$ or in $\overline{R}$, we can restrict to the case where $\deg(f) < \deg(g)$. Indeed, the containment $\overline{R} \subseteq W$ makes us sure that if $\frac{f}{g} \in R $, then $\deg(f) \leq \deg(g)$. If they have the same degree, then by Lemma \ref{hompart} we can write $f=h+f'$, $g=uh+g'$ where $u \in K^\times$ and $\deg (f) = \deg(h) > \max\{\deg(f'),\deg(g')\}$, and we can observe that in this case 
 $$ \frac{f}{g} \in R \Longleftrightarrow \frac{f}{g} - u^{-1} = \frac{f'-u^{-1}g'}{g} \in R.$$ This shows that we can restrict to the case where $\deg(f) < \deg(g)$.
 \end{rem}

The next lemma shows an important relation between the containment of primes of the form $\p_f$
and some asymptotic containment of principal ideals in $R$.
Given an ideal $I$, we denote its radical by $\sqrt{I}$.

   \begin{lemma}
\label{radicals}
 Let $D$ be an integral domain and $R=R(D)$.
Let $f,g  \in D$.
The following conditions are equivalent:
\begin{enumerate}
\item[$(1)$] $\p_g \subseteq \p_f$.
\item[$(2)$] $\sqrt{(\frac{1}{g})R} \subseteq \sqrt{(\frac{1}{f})R}$.
\item[$(3)$] There exists $e \geq 1$ such that $\frac{1}{g^e} \in (\frac{1}{f})R$ \rm(\it equivalently $ \frac{f}{g^e} \in R$\rm)\it.
\end{enumerate}
\end{lemma}
 
 \begin{proof}
The equivalence between (1) and (2) follows from the definition of radical as intersection of prime ideals, while the equivalence between (2) and (3) follows from the definition of radical in terms of powers of elements.
 \end{proof}
 
We analyze now the case when $R/\p$ is a DVR. By Theorem \ref{thm:p=px}, such $\p$ must be of the form $\p_f$ for some $f \in D$.

\begin{lemma}
\label{lemdvr}
Let $D$ be an E-simple domain with maximal subfield $K$. Let $f \in D \setminus K$ be such that $ R/\p_f$ is a DVR in which the image of $\frac{1}{f}$ is irreducible. 
Then, for $g \in D$, we have that $\frac{1}{g} \not \in \p_f$ if and only if $g \in K[f]$. In particular,
 $R(D)=R(K[f])+\p_f$ and $R(K[f]) \cap \p_f=(0)$.  
\end{lemma}

\begin{proof}
If $g \in K[f]$, we get $\frac{1}{g} \not \in \p_f$ by Lemma \ref{lemalgdip}.
For the other implication, assume that $\frac{1}{g} \not \in \p_f$. By Theorem \ref{thm:modp}, \rm $  R/\p_f \cong R(L) $ where $L$ is the subring of $D$ consisting of all the elements $h \in D$ such that $\frac{1}{h} \not \in \p_f$. Hence, $g \in L$.
By  Theorem \ref{thm:Euclconverse}, $L=K[h]$ where $h$ is any element such that the image of $\frac{1}{h}$ generates the maximal ideal of $R/\p_f$. By the assumptions on $f$, we must have $L=K[f]$ and therefore $g \in K[f]$. 

To see that $R(D) = R(K[f]) + \p_f$, let $0\neq \alpha \in R(D)$ and write $\alpha=\sum_{i=1}^n \frac 1{g_i}$. We may rearrange the terms so that $\frac 1{g_i} \notin \p_f$ for $1\leq i \leq t$ and $\frac 1{g_i} \in \p_f$ for $t+1 \leq i \leq n$, for some $0 \leq t \leq n$. Then by the first paragraph, we have $g_i \in K[f]$ for each $i \leq t$, whence $\alpha \in R(K[f]) + \p_f$.

To see that $R(K[f]) \cap \p_f = 0$, let $\alpha \in R(K[f]) \cap \p_f$. If $\alpha\neq 0$, then by Theorem~\ref{thm:Euclidean}, we have $\alpha = u f^{-s}$, where $s \geq 0$ and $u$ is a unit of $R(K[f])$.  Then in the composition $R(K[f]) \ra R \ra R/\p_f$, the image of $u$ is a unit and the image of $f^{-s}$ is nonzero; thus the image of $\alpha$ is nonzero, which contradicts the assumption that $\alpha \in \p_f$.
\end{proof}

  \begin{thm}
 \label{irredcondition}
Let $D$ be an E-simple domain with maximal subfield $K$ algebraically closed in $\Frac D$.
Let $g \in D \setminus K$ and $R=R(D)$.
Consider the following conditions:
\begin{enumerate}
\item[$(1)$] $ R/\p_g$ is a DVR in which the image of $\frac{1}{g}$ is irreducible. 
\item[$(1^*)$] For every $e \geq 1$, $ G_{\overline{R}}(g^e) = \langle 1,g, g^2, \ldots, g^e \rangle_K. $
\item[$(1^{**})$] For every $e \geq 1$, $ G(g^e) = \langle 1,g, g^2, \ldots, g^e \rangle_K. $
  \item[$(1^{***})$] $K[g]$ is a regular factroid of $D$. 
\item[$(2)$] $\frac{1}{g}$ is irreducible in $R$.
\item[$(3)$] $ G(g) = \langle 1,g \rangle_K. $ 
\item[$(4)$] $K[g]$ is a factroid of $D$.
\item[$(4')$] For any irreducible polynomial $p \in K[T]$, $p(g)$ is irreducible in $D$.
\item[$(4'')$] $ g+u$ is irreducible for every $u \in K$.
\item[$(5)$] $g$ is irreducible in $D$.
\end{enumerate}
Then $(1)\Rightarrow(2)\Rightarrow(3)\Rightarrow(4'')\Rightarrow(5),$   $(1) \Rightarrow (4) \Rightarrow (4') \Rightarrow (4'')$,  and  $(1) \Leftrightarrow (1^*) \Leftrightarrow (1^{**})   \Leftrightarrow (1^{***})$. In general, there are no other nontrivial implications between these conditions. However, if $D$ is a UFD, $(4') \Rightarrow (4)$, and if $K$ is algebraically closed, then $(4'') \Rightarrow (4')$. 
 \end{thm}

 To summarize: \[
 \xymatrix{
&(1^{***}) \ar@{=>}@/^1pc/[dr] & & (2) \ar@{=>}[r]&(3) \ar@{=>}@/^/[dr]\\
(1^{**}) \ar@{=>}@/^1pc/[ur]& & (1) \ar@{=>}@/^1pc/[dl] \ar@{=>}@/^/[ur] \ar@{=>}@/_/[dr] & & & (4'') \ar@{=>}[r] & (5).\\
&(1^*) \ar@{=>}@/^1pc/[ul]& & (4) \ar@{=>}[r] &(4') \ar@{=>}@/_/[ur]
 }
 \]

 \begin{proof}
$ (1)\Rightarrow(2): $ 
Suppose that $\frac{1}{g}$ is reducible in $R$ and
write it as $$ \frac{1}{g} = \left( \frac{1}{f_1} + \ldots + \frac{1}{f_s} \right) \left( \frac{1}{h_1} + \ldots + \frac{1}{h_t} \right) $$ with $f_i, h_j \in D \setminus K$. Since $\frac{1}{g} \not \in \p_g$, at least one $\frac{1}{f_i}$ and one $\frac{1}{h_j}$ are also not in $\p_g$.  Thus, after rearranging elements, there exist positive integers $m,n$ such that $\frac 1{f_1}, \ldots, \frac 1{f_m} \notin \p_g$ and $\frac 1{h_1}, \ldots, \frac 1{h_n} \notin \p_g$, whereas $\frac 1{f_{m+1}}, \ldots, \frac 1{f_s} \in \p_g$ and $\frac 1{h_{n+1}}, \ldots, \frac 1{h_t} \in \p_g$.  It follows that 
\[ \frac{1}{g} - \sum_{i=1}^m \sum_{j=1}^n \frac 1{f_i h_j} \in \p_g
\]
whereas $\frac{1}{f_ih_j} \notin \p_g$ whenever $1\leq i \leq m$, $1 \leq j \leq n$. 

Applying Lemma \ref{lemdvr}, we obtain $f_i h_j\in K[g]$ for all $i \leq m$, $j \leq n$, and thus
$\displaystyle \frac{1}{g} - \displaystyle \sum_{i=1}^m \sum_{j=1}^n \frac{1}{f_ih_j} \in R(K[g]) \cap \p_g = (0)$, where the vanishing also holds by Lemma~\ref{lemdvr}. Going modulo $\p_g$, we have that the image of $ \frac{1}{g} $ has value 1 in the DVR $ R/\p_g$, but the image of each of the terms $ \frac{1}{f_ih_j} $ has value at least 2 since $f_i, h_j \in K[g] \setminus K$
This yields a contradiction since $\displaystyle \frac{1}{g} \equiv \sum_{i=1}^m \sum_{j=1}^n\frac{1}{f_ih_j}$ mod $\p_g$. \\

$ (2)\Rightarrow(3): $ Pick $f \in G(g) \setminus K$. Hence, $ \frac{f}{g} \in R(D) $ and $ \frac{1}{g} = \frac{1}{f} \cdot \frac{f}{g}.  $ Since $ \frac{1}{g} $ is irreducible, we must have that $ \frac{f}{g} $ is a unit in $R(D)$. Hence, there exists a nonzero $u \in K$ such that $ \frac{f-ug}{g} = \frac{f}{g} - u \in \m. $ We then obtain $ \frac{1}{g} = \frac{1}{f-ug} \cdot \frac{f-ug}{g}.  $ The irreducibility of $ \frac{1}{g} $ forces $ \frac{1}{f-ug} $ to be a unit in $R(D)$, and therefore $f-ug \in K$. It follows that $f \in \langle 1,g \rangle_K.$ \\

$ (3)\Rightarrow(4''): $ It is clear since every proper factor of $g+u$ would be an element of the factroid $ G(g) $ that is not in $ \langle 1,g \rangle_K. $ \\

$ (4'')\Rightarrow(5): $ It is straightforward. \\

$ (1)\Rightarrow(1^*): $ Pick $f \in G_{\overline{R}}(g^e)$. We can consider an equation of integral dependence of $\frac{f}{g^e}$ over $R$, of degree $N$ for some $N \geq 1$.
Dividing such equation by $f^{N-1}$, we obtain  $\frac f{g^{eN}} \in R$.  By Lemma \ref{radicals}, it follows that $\p_g=\p_{g^{eN}} \subseteq \p_f$. Using the fact that $\dim R/\p_g =1$, we have two possibilities: if $\p_f = \m$, then $f \in K$ and the statement follows immediately. If $\p_f = \p_g \subsetneq \m$, then by Lemma \ref{lemdvr} we obtain $f \in K[g]$. Thus 
$ \frac{f}{g^e} \in \overline{R} \cap K(g) = R(K[g])$ (recall that $ R(K[g]) $ is a DVR contained in $R$ and $g \not \in \overline{R}$). The ring $A:=K[g]$ is graded and we have that $ \frac{f}{g^e} \in R(K[g])$ if and only if $\deg_A(f) \leq \deg_A(g^e)=e$. It follows that $f \in \langle 1,g, \ldots, g^e \rangle_K$. \\

$(1^*) \Rightarrow (1^{**})$: This is obvious. \\

$(1^{**}) \Rightarrow (1^{***})$: We have that $K[g] = \bigcup_e G(g^e)$ is a regular factroid of $D$ by Theorem~\ref{5}.\\

$(1^{***}) \Rightarrow (1)$: Since any regular factroid of $D$ containing $K[g]$ must contain $G(g^e)$ for every $e\in \N$, we have $K[g] = \bigcup_e G(g^e)$.  By Theorem~\ref{thm:Euclidean}, $R(K[g])$ is a DVR in which the image of $1/g$ is irreducible. But by Theorem~\ref{5}, the composition $R(L(\p_g)) \rightarrow R \ra R/\p_g$ is an isomorphism, and $L(\p_g) =  \bigcup_e G(g^e) = K[g]$, finishing the proof of (1).\\

$(1) \Rightarrow (4)$: It is obvious that $(1^{***}) \Rightarrow (4)$, and we showed above that $(1) \Leftrightarrow (1^{***})$.\\

$(4) \Rightarrow (4')$: Let $p\in K[T]$ be irreducible.  Let $c,d \in D$ such that $cd = p(g)$.  Since $p(g) \in K[g]$, which is a factroid of $D$, we have $c,d \in K[g]$.  Write $c=q(g)$, $d=r(g)$ for $q,r \in K[T]$.  Since the map $K[T] \ra D$ sending $T \mapsto g$ is injective (else $K$ would not be algebraically closed in $\Frac D$), it follows that $p=qr$ in $K[T]$.  By irreducibility, either $q\in K$ (in which case $c\in K$) or $r\in K$ (in which case $d\in K$).  Thus, $p(g)$ is irreducible in $D$.\\

$(4') \Rightarrow (4'')$: This follows because any linear polynomial in $K[T]$ is irreducible.\\

When $K$ is algebraically closed, $(4'') \Rightarrow (4')$: This is because every irreducible polynomial is of the form $T+u$.\\

When $D$ is a UFD, $(4') \implies (4)$: Let $c,d \in D$ with $cd \in K[g]$. Then there is some $p\in K[T]$ with $cd=p(g)$.  Write $p=q_1 \cdots q_r$, where each $q_i\in K[T]$ is an irreducible polynomial.  By assumption, $q_i(g)$ is irreducible in $D$ for all $i$. We have $cd = \prod_{i=1}^r q_i(g)$, so since $D$ is factorial with all units in $K$, after rearranging terms we may assume there is some $0\leq s \leq r$ and some $\lambda \in K^\times$ with $c = \lambda \prod_{i=1}^s q_i(g)$ and $d=\lambda^{-1}\prod_{i=s+1}^r q_i(g)$. Thus, $c, d\in K[g]$, completing the proof that $K[g]$ is a factroid of $D$.
 \end{proof}

 \begin{cor}
   \label{cor:elementsofK[X,Y]}
   Let $g \in (x,y)K[x,y] $ be an irreducible polynomial such that $\p_g \neq (0)$ in $R(K[x,y])$. Then $g$ satisfies all the conditions listed in Theorem \ref{irredcondition}. 
 \end{cor}

 \begin{proof}
     Combine Theorem \ref{quotientdim2} and Theorem \ref{irredcondition}.
 \end{proof}

 We   finish the proof of Theorem~\ref{irredcondition} by providing  a list of examples to show that   no other implications hold in general.
 
 \begin{example}
   $(5) \not\Rightarrow (4'')$: 
  The polynomial $f=xy+x+y \in D=K[x,y]$ is irreducible, but $xy+x+y+1 = (x+1)(y+1)$ is reducible. By Lemma \ref{lemma:associated}, we have 
 \[  \frac{1}{xy+x+y} = \theta \cdot \frac{1}{x} \cdot \frac{1}{y} \]
 where $\theta$ is a unit in $R(D)$. Thus $\frac 1 f$ is a reducible element of $R(D)$. 
 \end{example}

 \begin{example}  $(4'') \not\Rightarrow (3)$:  Consider $f = x^3+y^2+x^4y \in D=K[x,y]$.
 The polynomial $f + u $ is irreducible for every $u \in K$ (one can observe this by expressing it as a polynomial of degree 2 in the variable $y$). However,
 \[ \frac{1}{x^3+y^2+x^4y}= \frac{1}{y} \cdot \frac{1}{y+x^4} \left( \frac{1}{1 + \frac{x^3}{y^2+x^4y}}  \right).   \]
 We claim that the rightmost factor in the above formula is a unit in $R= R(D)$.  To see this, since $y,y+x^4 \in [y(y+x^4)]_D$, it follows that $x^4$, and hence its factor $x^3$, is also in $[y(y+x^4)]_D$. Thus,  $ \frac{x^3}{y^2+x^4y} \in R(D)$ by Theorem~\ref{thm:factroids}.  Then by the discussion following Definition~\ref{def:GAg}, since $\deg(x^3) = 3<5=\deg(y(y+x^4))$, it follows that $\frac{x^3}{y^2+x^4y} \in \m$, where $\m$ is the maximal ideal of $R$.  Then by Corollary~\ref{cor:units}, it follows that $\theta := 1+\frac{x^3}{y^2+x^4y}$ is a unit of $R$.  Thus, $\frac{y}{x^3 + y^2 + x^4y} = \theta^{-1} \cdot \frac 1{y+x^4}\in R$, so that by Theorem~\ref{thm:factroids}, $y \in G(f) \setminus \langle 1, f \rangle_K$.

 If we take $K$ to be algebraically closed, it also follows that (4) $\not \Rightarrow$ (3).
 \end{example}

\begin{example}  $(3) \not\Rightarrow (2)$:  Let $\kappa$ and $K$ be fields and let us use the symbols $x,y,z,x_1,x_2,x_3,x_4$ to denote some indeterminates.

Let $$ A= \kappa[x,y]\left[ \frac{xy}{x+y} \right] \cong \frac{\kappa[x,y,z]}{((x+y)z-xy)}.$$ Notice that setting $\deg(x)= \deg(y)= \deg(z) = 1$, $A$ is a graded $\kappa$-algebra where the degree zero component is just $\kappa$. Hence, $A$ is E-simple by Proposition \ref{pr:gradedarenotegyptian}.
The component of $A$ of degree one is given by $\langle x,y,z \rangle_K$.
By Remark \ref{remarkmember}, for every element of the form $\ell = u_0 + u_1 x + u_2y$ with $u_j \in \kappa$, we have $\frac{\ell}{z} \in R(A)$ if and only if $u_1,u_2=0$.

Consider now the $K$-algebra \begin{gather*}
    D= K[x_1,x_2,x_3,x_4]\left[ \frac{x_1x_2x_3x_4}{(x_1+x_2)(x_3+x_4)} \right] \cong \\ \frac{K[x_1,x_2,x_3,x_4,z]}{((x_1+x_2)(x_3+x_4)z-x_1x_2x_3x_4)}.
\end{gather*} Also $D$ is graded by setting $\deg(x_i)=1$ for $i=1,2,3,4$, and $\deg(z)=2$. Its degree zero component is $K$ and therefore also $D$ is E-simple.
In $R=R(D)$, we clearly have the factorization
$$ \frac{1}{z} = \left( \frac{1}{x_1} + \frac{1}{x_2}\right) \left( \frac{1}{x_3} + \frac{1}{x_4}\right)  $$ which makes $\frac{1}{z}$ not irreducible.
However, we claim that $G(z)= \langle 1, z \rangle_K$. By Remark \ref{remarkmember}, it is enough to show that no nonzero element of degree one of the form $\ell = \sum_{i=1}^4 u_j x_j $ with $u_j \in K$ is in $G(z)$.
Suppose some such $\ell \in G(z)$ and let $S= K[x_3,x_4]^{\circ}$. 
By Lemma \ref{lem:localizations} we have that
$$ T:= R(D)[x_3,x_4]= R \left( D \left[ \frac{1}{x_3}, \frac{1}{x_4} \right] \right)= R(S^{-1}D) \cong R(A),   $$ where the last isomorphism is obtained setting $x=x_1$, $y=x_2$, $\kappa=K(x_3,x_4)$, and rescaling $z$ by a unit. Since $G(z) \subseteq G_{T}(z)$, it follows that $u_1,u_2=0$. Similarly, we can consider $T'=R(D)[x_1,x_2]$ to obtain that $u_3,u_4=0$, which finishes the proof of the claim.
\end{example}

\begin{example}
$(2) \not\Rightarrow (1)$:  Consider the numerical semigroup algebra $D=K[x^2,x^3]$. This ring is E-simple with maximal subfield $K$. By \cite[Example 4.9]{Gu-recomp}, its reciprocal complement is $R(D)= K[x^{-2},x^{-3}]_{(x^{-2},x^{-3})}$, which is a Noetherian local one-dimensional domain, but is not regular.  We have \[x^3 \in [x^6]_D \subseteq G(x^6) = G((x^2)^3) \setminus \langle 1,(x^2), (x^2)^2, (x^2)^3\rangle_K,\] since $(x^3)^2 = x^6$, so (1**) is violated for $g=x^2$.
  However, the element $ \frac{1}{x^2} $ is irreducible in $R(D)$, so (2) holds for $g=x^2$.

Now let $K:=\F_7$ and $p(T):=T^3 - 4 \in K[T]$.  Since $4$ is not a cube mod $7$, $p$ does not have a root in $\F_7$, so that since $p$ has degree 3, it must be irreducible.  However, $p(x^2) = x^6-4 = (x^3-2)(x^3+2)$ in $D=K[x^2, x^3]$, so $x^2\in D$ does not satisfy (4$'$).  Thus, $(2) \not \Rightarrow (4')$.
\end{example}

\begin{example}
$(4') \not \Rightarrow (4)$:  Let $D = K[x,y,z]/(xy-z^2)$, with $K$ an algebraically closed field.  This is an E-simple domain since it is positively graded.  Moreover, for any $u\in K$, $z+u$ is irreducible in $D$.   However, $xy \in K[z]$ with $x \notin K[z]$ and $y \notin K[z]$, so $K[z]$ is not a factroid of $D$.  That is, $g=z$ satisfies (4$''$) (and thus (4$'$), since $K$ is algebraically closed), but not (4).
\end{example}

\begin{rem}
    Despite the fact that the implications $(1) \Rightarrow (2) \Rightarrow (3)$   and $(4') \Rightarrow (4'')$  are in general not reversible, we have not being able to find counterexamples for elements in a unique factorization domain. Whether they are equivalent or not over UFDs is still an open question.
\end{rem}

\subsection{A duality between regular factroid subrings and prime ideals}\label{sub:duality}
In this subsection, we establish a Galois correspondence (see e.g. \cite[Section 6.5]{Berg-genubook}) between the prime ideals of $R(D)$ and the subrings of $D$ that are regular factroids of $D$.  This strengthens Theorem~\ref{thm:p=px} by associating any prime ideal of $R(D)$ to a particular subset of $D$, even for primes $\p$ of $R(D)$ that have infinite Krull dimension.  This also puts Theorem~\ref{thm:modp} into clearer perspective.

\begin{defn}\label{def:theconnection}
For any multiplicative subset $W \subseteq D$, set $\p(W) := \displaystyle\bigcap_{f \in W^\circ} \p_f \subseteq R(D)$.

For any $S \subseteq R(D)$ that is a union of prime ideals, set $L(S) := \{0\neq x \in D \mid \frac 1x \notin S\} \cup \{0\} \subseteq D$.
\end{defn}

\begin{prop}\label{1}
    Let $W$ be a multiplicative submonoid of $D$ and $0\neq x \in D$.  Then \begin{enumerate}
        \item\label{it:pWprime} We have $\p(W) \in \Spec R(D)$, and
        \item\label{it:pWGW} $\frac 1x \notin \p(W)$ if and only if $x\in G(W)$.
    \end{enumerate}
\end{prop}

\begin{proof}
(\ref{it:pWprime}): It is clear that $\p(W)$ is an ideal of $R$.  To see that it is prime, let $\alpha, \beta \in R \setminus \p(W)$.  Then there exist $f,g \in D^\circ$ with $\alpha \notin \p_f$ and $\beta \notin \p_g$.  As it is clear that $\p_{fg} \subseteq \p_f \cap \p_g$, it follows that $\alpha \beta \notin \p_{fg}$.
But $fg\in W$ since the latter is multiplicatively closed, so $\alpha \beta \notin \p(W)$. \\

(\ref{it:pWGW}): First suppose $\frac 1x \notin \p(W)$.  Then there is some $f\in W$ with $\frac 1x \notin \p_f$, so by Lemma~\ref{radicals}, there is some $e\in \N$ with $\frac 1{f^e} \in \frac 1x R(D)$.  Write \[
\frac 1{f^e} = \frac 1x \left( \sum_{i=1}^\ell \frac 1{d_i}\right), \quad d_i \in D^\circ.
\]
Setting $d := \prod_{i=1}^\ell d_i$, and for each $i$, $\widehat{d_i} := d/d_i \in D$, we have $dx = \sum_i \widehat{d_i} f^e$.  Since each $\widehat{d_i}$ is a factor of $d$, we have $dx \in [df^e]_D$.  Thus, $x \in G(f^e) \subseteq G(W)$.

Conversely, suppose $x\in G(W)$.  Then for some $d\in D^\circ$, $dx \in [dW]_D$. That is, $dx = \sum_{i=1}^\ell c_i$, $c_i b_i =d w_i$ for $b_i, c_i \in D^\circ$ and $w_i \in W$.  Hence, $x=\sum_{i=1}^\ell w_i / b_i$, and without loss of generality each $w_i\neq 0$.  Setting $w = \prod_i w_i$, and for each $i$, $\widehat{w_i} := w/w_i \in W$, it follows that \[
\frac 1w = \frac 1x \left(\sum_{i=1}^\ell \frac 1 {b_i \widehat{w_i}}\right) \in \frac 1x R(D).
\]
Hence, $\frac 1x \notin \p_w$, whence $\frac 1x \notin \p(W)$.
\end{proof}

\begin{prop}\label{4}
For any union $S = \bigcup_{\q \in X} \q$ of primes of $R(D)$, $L(S)$ is both a subring and a regular factroid of $D$.
\end{prop}

\begin{proof}
Let $x, y \in L(S) \setminus \{0\}$.  Then $\frac 1x, \frac 1y \notin S$, so since $R \setminus S$ is multiplicatively closed, $\frac 1{xy} \notin S$. Thus, $xy \in L(S)$.  Furthermore, suppose $x+y \neq 0$.  Then $\frac 1{xy} = \frac 1{x+y} \left( \frac 1x + \frac 1y\right)$. Since $R \setminus S$ is a saturated multiplicative set and $\frac 1{xy} \notin S$, it then follows that $\frac 1{x+y} \notin S$, so that $x+y \in L(S)$.  Since also $1=\frac 11 \notin S$, so that $1\in L(S)$, this completes the subring claim.

To see that it is a factroid, let $c,d \in D^\circ$ such that $cd \in L(S)$.  Then for all $\q \in X$, $\frac 1c \frac 1d = \frac 1{cd} \notin \q$.  Thus $\frac 1d \notin \q$ for all $\q$, so $\frac 1d \notin S$, whence $d\in L(S)$.

Finally, we show that the factroid in question is regular.  Let $0\neq x \in G(L(S))$.  Then $\frac 1x\notin \p(L(S))$ by Proposition~\ref{1}, so there is some $0 \neq f\in L(S)$ such that $\frac 1x \notin \p_f$. But then by Lemma~\ref{radicals}, $\frac 1f \in \sqrt{\frac 1x R(D)}$.  Since $f\in L(S)$, we have $\frac 1f \notin S$, so $\frac 1f \in \sqrt{\frac 1x R} \setminus S$, whence $\sqrt{\frac 1x R} \nsubseteq S$.  If $\frac 1x\in S$, then there is some $\q \in X$, $\frac 1x \in \q$, so $\sqrt{\frac 1x R} \subseteq \q \subseteq S$, a contradiction. Thus, $\frac 1x \notin S$, so $x\in L(S)$.  Thus, $L(S) = G(L(S))$, completing the proof of regularity.
\end{proof}

\begin{thm}\label{5}
Let $D$ be an integral domain and $R := R(D)$. Let $\cM = $ the poset of multiplicative submonoids of $D$, and $\cU = $ the poset of unions of prime ideals of $R$.  The functions $\p: \cM \ra \cU$ and $L: \cU \ra \cM$ and $L(-)$ from Definition~\ref{def:theconnection} satisfy the following properties: \begin{enumerate}[(a)]
    \item\label{it:containments} For any $S \in \cU$ and $W \in \cM$, we have $W \subseteq L(S) \iff S \subseteq \p(W)$.
    \item\label{it:gc} We have an \emph{antitone Galois connection} between $\p(-)$ and $L(-)$.
    \item\label{it:LpG} For any $W \in \cM$, $L(\p(W)) = G(W)$.
    \item\label{it:LpfGfe} For any $f \in D^\circ$, $L(\p_f) = \bigcup_{e\in \N} G(f^e)$.
    \item\label{it:pLq} For any $\q \in \Spec R$, we have $\p(L(\q))=\q$.
    \item\label{it:LpC} For any regular factroid subring $C$ of $D$, we have $C = L(\p(C))$.
\end{enumerate}
Hence, $(\p(-), L(-))$ restricts to a \emph{Galois correspondence} (and hence, a one-to-one order-reversing correspondence) between the \emph{regular factroid subrings} $C$ of $D$ and the \emph{prime ideals} $\q$ of $R$.

Moreover, if $(C,\q)$ is such a pair, then the induced composition of natural ring homomorphisms $R(C) \ra R \ra R/\q$ is an isomorphism.
\end{thm}

\begin{proof}
\ref{it:containments}: Let $S \in \cU$ and $W \in \cM$. Then $S = \displaystyle \bigcup_{\q \in X} \q$ for some subset $X \subseteq \Spec R$.

Suppose $W \subseteq L(S)$.  Since $L(S)$ is a regular factroid of $D$ by Proposition \ref{4}, it follows that $G(W) \subseteq L(S)$.  Let $\alpha \in S$.  Then there is some $\q \in X$ with $\alpha \in \q$. By Proposition \ref{pr:decomposeelement}, we have $\alpha = \sum_{i=1}^t \frac 1{x_i}$ for some $x_i \in D^\circ$ such that $\frac 1{x_i} \in \q$.  But then $x_i \notin L(S)$, whence $x_i \notin G(W)$, so by Proposition \ref{1}, $\frac 1{x_i} \in \p(W)$.  Since the latter is an ideal, it follows that $\alpha = \sum_i 1/x_i \in \p(W)$. Thus, $S \subseteq \p(W)$.

Conversely, suppose $S \subseteq \p(W)$.  Let $f \in W$.  If $f \notin L(S)$, then $\frac 1f \in S$, so that since $S \subseteq \p(W)$, we have $\frac 1f \in \p(W) \subseteq \p_f$, a contradiction.  Hence, $f\in L(S)$, so that $W \subseteq L(S)$. \\

\ref{it:gc}: It suffices to show that $L(-)$ and $\p(-)$ are order-reversing.
Suppose $S \subseteq T$ are unions of prime ideals of $R(D)$. Let $0 \neq x\in L(T)$.  Then $\frac 1x \notin T$, so $\frac 1x \notin S$ since $S \subseteq T$, whence $x\in L(S)$. Thus, $L(T) \subseteq L(S)$.
Now suppose $V \subseteq W \subseteq D$ are multiplicative submonoids.  Then $\p(V) = \bigcap_{f \in V} \p_f \supseteq \bigcap_{f \in W} \p_f = \p(W)$. \\

\ref{it:LpG}: By \ref{it:containments}, $W \subseteq L(\p(W))$, which by Proposition~\ref{4} is a regular factroid of $D$. Hence, $G(W) \subseteq L(\p(W))$.  Conversely, let $0\neq t \in L(\p(W))$. Then $\frac 1t \notin \p(W)$, so by  Proposition~\ref{1}(2), $t \in G(W)$. \\

\ref{it:LpfGfe}: Let $f\in D^\circ$, and set $W := \{f^e \mid e\in \N_0\}$.  Then for any $e\geq 1$, $\p_{f^e}=\p_f$, and $\p_{f^0}=\p_1=\m_{R}$, so $\bigcap_{x\in W} \p_x = \bigcap_{e\geq 1} \p_{f^e} = \p_f$.  Thus by \ref{it:LpG}, $L(\p_f) = G(W)$.  Clearly $G(f^e) \subseteq G(W)$ for each $e$.  Conversely, for any $0 \neq x \in L(\p_f)$, we have $\frac 1x \notin \p_f$, so by Lemma~\ref{radicals}, $\frac 1{f^e} \in \frac 1x R$ for some $e\in \N$, which means that $\frac x{f^e} \in R$, so by Theorem~\ref{thm:factroids}, $x\in G(f^e)$.  Thus, $G(W) = \bigcup_e G(f^e)$. \\

\ref{it:pLq}: The containment $\q \subseteq \p(L(\q))$ follows from \ref{it:containments}. 
For the opposite containment, observe that $\p(L(\q)) = \bigcap_{f \in L(\q)} \p_f = \bigcap\{\p_f \mid f \in D^\circ,\ \frac 1f \notin \q\}$.  Let $\alpha \in R \setminus \q$. Write $\alpha = \sum_{i=1}^\ell \frac 1{x_i}$, $x_i \in D^\circ$. By rearranging terms if necessary, we have $\frac 1{x_i} \notin \q$ for $1\leq i \leq m$, $\frac 1{x_i} \in \q$ for $m+1 \leq i \leq \ell$, and $m\geq 1$.  Set $\alpha' := \sum_{i=1}^m \frac 1{x_i}$ and $\beta := \alpha - \alpha' = \sum_{i=m+1}^\ell \frac 1{x_i}$.  For each $i\leq m$, we have $\frac 1{x_i} \notin \p_{x_i}$.  Set $x := \prod_{i=1}^m x_i$. Then for each $i \leq m$, $\p_{x_i} \supseteq \p_x$, so $\frac 1{x_i} \notin \p_x$, and $\frac 1x = \prod_{i=1}^m \frac 1{x_i} \notin \q$, whence $\p_x \supseteq \p(L(\q))$.  Set $y := \sum_{i=1}^m \prod_{j\neq i, j\leq m} x_j$.  We have $\alpha' = y/x$.  Since $\alpha' \neq 0$ (otherwise $\alpha=\beta \in \q$), we have $y\neq 0$.  Thus, $\frac 1y \alpha' = \frac 1x \notin \p_x$, so $\alpha' \notin \p_x$.  Since $\beta \in \q \subseteq \p_x$, we have $\alpha = \alpha'+\beta \notin \p_x$. Thus, $\alpha \notin \p(L(\q))$. \\

\ref{it:LpC}: An application of \ref{it:containments} shows that $C \subseteq L(\p(C))$.
Conversely, let $0\neq x\in L(\p(C))$.  Then $\frac 1x \notin \p(C)$, so by Proposition~\ref{1}(2), $x \in G(C) = C$, with equality by assumption on $C$. \\

The final statement follows from Theorem~\ref{thm:modp}.
\end{proof}

\begin{cor}
Any nonzero prime of $R(D)$ is an intersection of ideals of the form $\p_f$, $f \in D^\circ$.
\end{cor}

\section*{Acknowledgments} 
We are grateful for many enlightening discussions on the material of this paper with Alan Loper and Dario Spirito. In the broader community, we are thankful to Justin Chen and Will Sawin for their help in establishing the distinction between E-simplicity and unit additivity (see Example~\ref{ex:ChenSawin}).

\providecommand{\bysame}{\leavevmode\hbox to3em{\hrulefill}\thinspace}
\providecommand{\MR}{\relax\ifhmode\unskip\space\fi MR }
\providecommand{\MRhref}[2]{%
  \href{http://www.ams.org/mathscinet-getitem?mr=#1}{#2}
}
\providecommand{\href}[2]{#2}


\begin{thebibliography}{GLO26}

\bibitem[Ber15]{Berg-genubook}
George~M. Bergman, \emph{An invitation to general algebra and universal
  constructions}, second ed., Universitext, Springer, Cham, 2015.

\bibitem[BG09]{BrGu-polybook}
Winfried Bruns and Joseph Gubeladze, \emph{Polytopes, rings, and {$K$}-theory},
  Springer Monographs in Mathematics, Springer, Dordrecht, 2009.

\bibitem[Che25]{Ch-pers25}
Justin Chen, personal communication, 2025.

\bibitem[EE26]{Elnme-factroids}
Jesse Elliott and Neil Epstein, \emph{Additive subgroups of a module that are
  saturated with respect to a subset of the ring}, The ideal theory and
  arithmetic of rings, monoids, and semigroups, Contemp. Math., vol. 836, Amer.
  Math. Soc., Providence, RI, 2026, pp.~195--231.

\bibitem[EGL25]{nmeGuLo-poly}
Neil Epstein, Lorenzo Guerrieri, and K.~Alan Loper, \emph{The reciprocal
  complement of a polynomial ring in several variables over a field}, Pacific
  J. Math. \textbf{338} (2025), no.~2, 267--293.

\bibitem[Eps24]{nme-Edom}
Neil Epstein, \emph{Egyptian integral domains}, Ric. Mat. \textbf{73} (2024),
  no.~5, 2901--2910.

\bibitem[Eps25]{nme-Euclidean}
\bysame, \emph{The unit fractions from a {E}uclidean domain generate a {DVR}},
  Ric. Mat. \textbf{74} (2025), no.~4, 2307--2313.

\bibitem[ES25]{nmeSh-unitadd}
Neil Epstein and Jay Shapiro, \emph{Rings where a non-nilpotent sum of units is
  a unit}, J. Algebra \textbf{672} (2025), 120--144, {\em erratum} {\bf 679}
  (2025), 65--66.

\bibitem[GLO26]{GLO-Egypt}
Lorenzo Guerrieri, K.~Alan Loper, and Greg Oman, \emph{From ancient {E}gyptian
  fractions to modern algebra}, J. Algebra Appl. \textbf{25} (2026), no.~8,
  Paper No. 2650078, 17 pp.

\bibitem[Gue25]{Gu-recomp}
Lorenzo Guerrieri, \emph{The reciprocal complements of classes of integral
  domains}, J. Algebra \textbf{682} (2025), 188--214.

\bibitem[Kap62]{Kap-Rseq}
Irving Kaplansky, \emph{{$R$}-sequences and homological dimension}, Nagoya
  Math. J. \textbf{20} (1962), 195--199.

\bibitem[Kap70]{Kap-CR}
\bysame, \emph{Commutative rings}, Allyn and Bacon Inc., Boston, 1970.

\bibitem[Saw]{Saw-MOthisring}
Will Sawin, \emph{What is the group of units of this ring?}, MathOverflow,
  URL:https://mathoverflow.net/q/502075 (version: 2025-10-25).

\bibitem[Spi25]{Spi-recocurve}
Dario Spirito, \emph{The reciprocal complement of a curve}, Ann. Mat. Pura
  Appl. (4) (2025), 6 pages, doi:10.1007/s10231-025-01617-5.

\end{thebibliography}
\end{document}